\documentclass[11pt]{amsproc}

\usepackage{mathtext}
\usepackage[cp1251]{inputenc}

\usepackage{bm}

\usepackage{amsmath}
\usepackage{amssymb}
\usepackage{amsxtra}

\def\N{{{\Bbb N}}}
\def\Z{{{\Bbb Z}}}
\def\T{{{\Bbb T}}}
\def\R{{\Bbb R}}

\def\l{{\lambda }}
\def\a{{\alpha }}
\def\D{{\Delta }}

\def\a{{\alpha}}
\def\b{{\beta}}
\def\d{{\delta}}

\def\L{{\Lambda }}

\def\){\right)}
\def\({\left(}

\numberwithin{equation}{section}

\newtheorem{corollary}{Corollary}[section]
\newtheorem{lemma}{Lemma}[section]
\newtheorem{theorem}{Theorem}[section]

\newtheorem{remark}{Remark}[section]

\renewcommand{\k}{\bm{k}}
\def\n{\bm{n}}
\def\x{\bm{x}}
\def\m{\bm{m}}

\par

\begin{document}

\title[Asymptotics of the Lebesgue constant]{Asymptotics of the Lebesgue constants \\ for a $d$-dimensional simplex}

\author[Yurii Kolomoitsev]{Yurii Kolomoitsev$^{\text{a, b, *, 1}}$}
\address{Universit\"at zu L\"ubeck, Institut f\"ur Mathematik,
Ratzeburger Allee 160, 23562 L\"ubeck, Germany}
\email{kolomoitsev@math.uni-luebeck.de}

\author[Elijah Liflyand]{Elijah Liflyand$^{\text{c}}$}
\address{Department of Mathematics, Bar-Ilan University, 52900 Ramat-Gan, Israel}
\email{liflyand@math.biu.ac.il}

\thanks{$^\text{a}$Universit\"at zu L\"ubeck,
Institut f\"ur Mathematik,
Ratzeburger Allee 160,
23562 L\"ubeck, Germany}

\thanks{$^\text{b}$Institute of Applied Mathematics and Mechanics of NAS of Ukraine,
General Batyuk Str. 19, Slov’yans’k, Donetsk region, Ukraine, 84100}

\thanks{$^\text{c}${Department of Mathematics, Bar-Ilan University, 52900 Ramat-Gan, Israel}}

\thanks{$^1$Supported by DFG project KO 5804/1-1.}

\thanks{$^*$Corresponding author}


\date{\today}

\begin{abstract}
In this paper an asymptotic formula is given for the Lebesgue constants generated by the anisotropically dilated $d$-dimensional simplex.
Contrary to many preceding results established only in dimension two, the obtained ones are proved in any dimension.
Also, the ``rational" and ``irrational" parts are both united and separated in one formula.
\end{abstract}

\maketitle

\section{Introduction}

In the theory of multiple Fourier series, many properties depend on or are given in terms of the Lebesgue constants, the
$L_1$ (or continuous) norms of the operators of taking partial sums. More precisely, let $\T^d \simeq (-\pi,\pi]^d$, $d=2,3,\dots$, be the
$d$-dimensional torus 
and
$$
\Vert f\Vert_{(d)}=\Vert f\Vert_{L_1(\T^d)}=\int_{\T^d}|f({\bm x})|\,{\rm d}{\bm x}.
$$
For the Fourier series of a function $f\in L_1(\mathbb T^d)$
$$
f(\x)\sim\sum\limits_{\bm{k}\in\mathbb Z^d}\widehat{f}(\bm{k})e^{i(\bm{x},\,\bm{k})},
$$
where
$(\bm{x},\bm{k})=k_1x_1+\dots+k_dx_d$ and
$
\widehat{f}(\bm{k})=(2\pi)^{-d} \int_{\T^d} f(\x)e^{-i(\bm{x},\bm{k})}\,{\rm d}{\bm x}
$
are the Fourier coefficients, we consider its partial sums generated by a set $W$, that is,

$$
\sum\limits_{\bm{k}\in W\cap \Z^d}\widehat{f}(\k)e^{i(\bm{x},\,\bm{k})},
$$
taking $L_1(\mathbb T^d)$ into $L_1(\mathbb T^d)$ (or $C(\mathbb T^d)$ into
$C(\mathbb T^d)$). As is well known, in most natural cases these norms are equal to

$$
\mathcal{L}_W=(2\pi)^{-d} \int_{\T^d} |\mathcal{D}_W({\bm x})|\,{\rm d}{\bm x},
$$
where
$$
\mathcal{D}_W({\bm x})=\sum\limits_{\bm{k}\in\mathbb Z^d\cap W}e^{i(\bm{x},\,\bm{k})}
$$
is the Dirichlet kernel generated by $W$. Correspondingly, $\mathcal{L}_W$ is called the Lebesgue constant
generated by $W$. One of the main features of the multivariate ($d\ge2$) theory of series
is a variety of options for ordering partial sums. In our terms, we can consider various $W$, whose geometrical
and arithmetical properties make $\mathcal{L}_W$ to behave in a diversity of ways. Information on this behavior is of
importance, sometimes crucial, in many problems of summability of Fourier series, approximation, and so on.
The reader can find a comprehensive survey of the results for various Lebesgue constants in \cite{L} (see also~\cite[Ch.~9]{TB}), where
overviewed are not only problems for the partial sums but also for the linear means of Fourier series.
In a recent paper \cite{GaLi}, just the results on the Lebesgue constants for partial sums are collected
in a condensed form.

\subsection{Specifics and history}

One of the ways to feel what kind of difficulties one faces in the multivariate case
is to compare it with the one-dimensional background, where

$$\frac1{2\pi}\int_{\mathbb T}\biggl|\sum\limits_{k=0}^n e^{ikx}\biggr|\,{\rm d}x=
\frac1{2\pi}\int_{\mathbb T}\biggl|\frac{\sin\frac{(n+1)x}2}{\sin\frac{x}2}\biggr|\,{\rm d}x=\frac4{\pi^2}\log n+O(1).$$
This nice classical result comes from the very definite ordering, which, in turn, allows one to aggregate the Dirichlet kernel.
In several dimensions, the kernel can be aggregated in a very few and mostly trivial situations.

Among various ways to distinguish between different $W$ and their Lebesgue constants, we mention the one very simple
but very determining. One type of the Lebesgue constants is formed by $W$ having at least one point with
nonzero curvature at their boundary. The other one involves only polyhedral $W$, where no points with nonzero
curvature occur. In this paper, we will deal with the Lebesgue constants of the second, ``polyhedral" type. Roughly speaking,
the Lebesgue constants of the former type are of ``spherical", exponential nature, while those of the latter
type are of ``logarithmic" nature. Each type has its own applications; in addition to those mentioned general,
we make a point of a very special application to the multivariate polynomial
interpolation on the Lissajous–Chebyshev nodes in~\cite{DEKL}.

Let us point out basic landmarks in the history of the polyhedral case. There are two main directions in defining the
Dirichlet kernels and corresponding Lebesgue constants. One is related with dilations of the fixed $W$, say, $nW$,
and the study of dependence of the Lebesgue constants on $n$. Here Belinsky's result \cite{Be} (see also \cite{Ba}, \cite{P};
for $L_p$ norms, with $p>1$, see \cite{Ash} and \cite{AC}) is of notable importance. It reads as follows.

\medskip

{\bf Theorem A.} {\it For any convex $d$-dimensional polyhedron $W\subset\mathbb R^d$, there exist two
positive constants $C_1$ and $C_2$ such that for any $n\ge1$}
\begin{equation}\label{belbi}
C_1\log^d(n+1)\le \mathcal{L}_{nW}\le C_2\log^d(n+1).
\end{equation}

More intriguing is the case where $\mathbb Z^d$ is exhausted in
a ``chaotic" way, that is, the rates of growth of the parameters in each direction are independent one of another.

\medskip

For both types, problems of asymptotic behavior of the sequence of the Lebesgue constants are central and mostly open.
Probably, the first nontrivial result of such kind was due to Daugavet \cite{Daug}. It had become clear soon that
asymptotic results are subject to the undertaken efforts only in the case where all the slopes of the sides of a
polyhedral set are rational. Kuznetsova's asymptotic formula \cite{Ku} for rhombic Lebesgue constants, with integer ratio of the sides running
to infinity (as an extension of the case of equal growth in \cite{Daug}) for a long time continued to be the best
possible asymptotic relation, though being proved, as most of others, in dimension two only.
A landmark is the following result by Skopina \cite{S} (see also \cite{KS})
for the $n$-dilations of the simplex.

\medskip

{\bf Theorem B.} 
{\it Let $\bm{\a}=(\a_1,\dots,\a_d)$ be a vector with positive rational coordinates.   Then,  for large values of $n$,  we have}
\begin{equation}\label{eS}
 \mathcal{L}_{n\D(\bm{\a})} =\frac{2^{d+1}(d+1)}{\pi}\log^d n+\mathcal{O}\(\log^{d-1} n\),
\end{equation}
{\it where}
$$
\D({\bm{\a}}):=\left\{\bm{\xi}\in \R_+^d\,:\,\sum_{j=1}^d\frac{\xi_j}{\a_j}\le 1\right\}
$$
{\it is a $d$-dimensional simplex and the constant in $\mathcal{O}$ depends only on $\bm{\a}$ and $d$.}

\medskip

Positive asymptotic results for rational slopes led to the natural questions about irrational ones. Started by Podkorytov in \cite{P2},
the study of this problem was continued and essentially developed in his joined work with Nazarov \cite{NP}. The main
conclusion is that even one irrational slope results, in general, in absence of asymptotics. The study of exactly when
this is the case, as well as of the possible upper and lower limits seriously involves number theory methods and is
far from being final. It would not be out of place to say that these results are, as many others, only two-dimensional.
Denoting $W_\a=\{(u,v)\,:\, 0\le v\le 1,\,\, 0\le v\le \a u\}$ and $I_n(\a)=\int_{-\pi}^\pi |\sum_{0\le k\le n}\{\a k\}e^{ikx}|\,{\rm d}x$,
where $\{\cdot\}$ is the fractional part, we present their main results as follows.

\medskip

{\bf Theorem C.} {\it For large values of $n$, one has}
$$
\mathcal{L}_{nW_\a}=\frac{24}{\pi}\log^2 n+2\pi I_n(\a)+\mathcal{O}(\log n\log\log n),
$$
{\it where the term $I_n(\a)$ for $\a\not\in \mathbb{Q}$ has the following behavior:}

1) $0<c\le \operatornamewithlimits{\overline{\lim}}\limits_{n\to\infty}\frac{I_n(\a)}{\log^2 n}\le C$;

2) {\it $\operatornamewithlimits{\underline{\lim}}\limits_{n\to\infty}\frac{I_n(\a)}{\log^2 n}=0$ if and only if $\a$ is a Liouville number;}

3) {\it if $|\a-\frac pq|\le \frac1{q^M}$ for some $M>2$ and infinitely many fractions $\frac pq$, $q\ge 2$, then the fraction $\frac{I_n(\a)}{\log^2 n}$ does not have a limit as $n\to \infty$;}

4) {\it there exist the numbers $\omega$ and $\Omega$ such that $0<\omega\le \Omega<\infty$ and for almost all $\a$}
$$
\omega=\operatornamewithlimits{\underline{\lim}}\limits_{n\to\infty}\frac{I_n(\a)}{\log^2 n}\quad\text{and}\quad
\Omega=\operatornamewithlimits{\overline{\lim}}\limits_{n\to\infty}\frac{I_n(\a)}{\log^2 n}.
$$

\medskip

Back to bilateral estimates of the polyhedral Lebesgue constants, in
\cite{YuYu} the constant in the upper logarithmic estimate was made proportional to the number of sides.
Recently, in \cite{KL} such an estimate was improved to be proportional to the number of triangulations.
This required delicate methods of number theory; in addition, unlike the preceding two-dimensional estimates,
those in \cite{KL} are established for arbitrary dimension.

There was a period when an impression of stagnation in the topic had happened to be.
For example, in \cite{SV} and \cite{DMV}, the well known estimates are reproved.
What has somehow got the topic off the ground were apparently \cite{DEKL} and~\cite{KL}.
All these is a background we started our study with.

\subsection{Goals and structure}

In this work, we deal with simplices from the first hyper-octant supported by the coordinate hyper-planes.
They are dilated anisotropically. We have obtained an asymptotic type formula by splitting our Lebesgue constants into two parts:
a kind of tame one and the rest, which can be ``good" or ``bad" depending on arithmetic properties of $W$.
The (achieved) goals of this work are:
\begin{enumerate}
\item Theorem B is generalized to the case of anisotropic dilations.

\item Kuznetsova's type result is improved and extended to arbitrary dimension.

\item Possibly problematic ``irrational" part is separated (like in Theorem C) and in certain cases is estimated in an appropriate way.

\end{enumerate}

To give a flavor of the obtained relations (Theorem \ref{mar}), we present their two-dimensional version so far:
{\it Let $\n=(n_1,n_2)\in \R^2$ be such that $n_2\ge n_1>3$. Then}
\begin{equation*}\label{asymd2}
\begin{split}
    \mathcal{L}_{\D(\n)}=\frac {16}\pi\log n_1\log n_2+\frac{8}{\pi}\log^2 n_1&+2\pi \int_{-\pi}^\pi
    \bigg|\sum_{0\le k\le n_1}\Big\{ n_2\Big(1-\frac{k_1}{n_1}\Big)\Big\} e^{ikx}\bigg|\,{\rm d}x \\
    &+ \mathcal{O}\(\log\log n_1 \log n_2\).
\end{split}
\end{equation*}
{\it Moreover, if $\frac{\log n_2}{\log n_1}\to \infty$ as $n_1\to \infty$, then}
\begin{equation*}\label{asymd2}
\begin{split}
    \mathcal{L}_{\D(\n)}=\frac {16}\pi\log n_1\log n_2+ \mathcal{O}\(\log^2 n_1+\log\log n_1 \log n_2 \).
\end{split}
\end{equation*}

It is worth noting that in the above asymptotic formula,  the irrational properties of the numbers $n_1$ and $n_2$ play no role,
unlike in the corresponding isotropic result studied by Nazarov and Podkorytov, cf. first of all 3) in Theorem C above.

One more remark is in order. The cases of $d$-dimensional rhombus or more general polyhedra, with the exception of the square,
are even more complicated. The authors hope that their approach is universal enough to crack these nuts as well; further results
will appear elsewhere.

Concentrated, our work is structured as follows. As mentioned, much attention and place is given to notation.
This is mainly done in the next section. With most of notation in hand, we are able to immediately formulate
in the same section our main results. The last section is devoted to auxiliary statements, their proofs and the proofs
of the main results. We mention that among auxiliary results, Lemma \ref{main} plays a pivotal role.

\section{Notations and main results}

As usual in the multidimensional case, a lot depends on convenient and effective notation.
Since there is much of it, we will introduce it by portions along with certain preliminary remarks.

\subsection{Basic notations and preliminary remarks}

In addition to the introduced, we will use the following notation:
$
\bm{k}^s=(k_1,\dots,k_s)\in \Z^{s},
$
$
\bm{\xi}^s=(\xi_1,\dots,\xi_s)\in \R^s,
$
and, as above, $(\bm{x}^s,\bm{k}^s)=k_1x_1+\dots+k_sx_s$, where $s\in \N$, $s\le d$, and $d\in \{2,3,\dots\}$. By $\mathbf{1}^s$, we understand the $s$-dimensional vector
$(1,1,\dots,1)$, and by $\{{\rm e}_j^s\}_{j=1}^s$ we denote the standard basis in $\R^s$.  Denoting
$
\bm{n}^s=(n_1,\dots,n_s)\in \R^s
$  and
$$
\m^{(s)}:=\(\frac{n_{s+1}}{n_1},\dots,\frac{n_{s+1}}{n_s}\)\in \R^s,
$$
we associate with them the vector function
$
\bm{\L}^s=(\L_1,\dots,\L_s)\,:\, \R^{s-1}\mapsto \R^{s},
$
where
$
\L_1=n_1
$
and
\begin{equation}\label{FORML}
  \L_s=\L_s(\bm{\xi}^{s-1}):=n_s-(\bm{m}^{(s-1)},\bm{\xi}^{s-1}),\quad s=2,3\dots.
\end{equation}
For simplicity, we shall frequently use the notation
$\n_2^s=(n_2,\dots,n_s)\in \R^{s-1}$, 
$$
\n=\n^d,\quad \n_2=\n_2^d,\quad \n'=\n^{d-1},\quad  \n_2'=\n_2^{d-1},\quad \text{and} \quad \frac1{\bm{n}}=\(\frac1{n_1},\dots,\frac1{n_d}\);
$$
and similarly for the vectors $\x^d$, $\k^d$, $\bm{\xi}^d$, etc. More precisely, $\n'$ means that the last coordinate is omitted.
We would like to emphasize that in what follows we {\it do not assume the vectors $\n^s$, $s=1,\dots,d$, to be necessarily with integer coordinates}.
The floor and the fractional part functions are defined, as usual, by
$[x] =\max(m\in \Z\,:\, m\le x)$  \ {and}\  $\{x\}=x-[x],$ correspondingly.

Throughout the paper, we suppose that $\prod\limits_{{\bm k}\in \varnothing}=1$, $\sum\limits_{{\bm k}\in \varnothing}=0$, and $\sum\limits_{k=A}^B(\dots)=0$ if $A>B$.
We will also make use of the following auxiliary notation:
$$
\Pi(\n):=\prod_{j=1}^d \log n_j.
$$
Also, we use the notation $\, A \lesssim B$ for the
estimate $\, A \le c\, B$, where $A$ and $B$ are some nonnegative functions and $c$ is a positive constant independent of
the relevant variables in $\, A$ and $\, B.$ Such a constant $c$ may depend only on the dimension $d$.

\subsection{More notations and main results}

One of the main objects of this paper is the following Dirichlet type kernel related to the simplex $\D_{\bm{n}}$:
$$
D_{ \bm{n}}(\bm{x}):=\sum_{\bm{k}=0}^{[\bm{\L}^d]}e^{i(\bm{k},\,\bm{x})}=\sum_{k_1=0}^{[\L_1]}
\sum_{k_2=0}^{[\L_2(k_1)]}\dots \sum_{k_d=0}^{[\L_d(\bm{k}')]}e^{i(\bm{k},\,\bm{x})},
$$
where $\bm{\L}^d=(\L_1,\dots,\L_d)$ is given by~\eqref{FORML}. To match the notation, we note that $D_{ \bm{n}}=\mathcal{D}_{\D(\n)}$.
There is also a need in the following notation:
$$
F_{\bm{n}}(\bm{x}'):=\sum_{\bm{k}'=0}^{[\bm{\L}^{d-1}]}\{\L_d(\bm{k}')\}e^{i(\bm{k}',\,\bm{x}')};
$$
by $\d_{h,\bm{\eta}^{s}}$, where $h\in \R$ and $\bm{\xi}^s\in \R^s$, we denote the operator acting on a function $f(\x^s)$ by the rule
$$
\d_{h,\bm{\xi}^{s}} f(\x^s):=e^{ih}f(\x^s-h \bm{\xi}^{s})-f(\x^s);
$$
further, we denote
\begin{equation}\label{e2}
  \begin{split}
    S_{\bm{n}}(\bm{x})&:=\sum_{\bm{k}'=0}^{[\bm{\L}^{d-1}]}e^{i(\bm{k}',\,\bm{x}')}\frac{e^{i\L_d(\bm{k}')x_d}-1}{ix_d}=\frac1{ix_d}\d_{n_d x_d,\frac1{\n'}}D_{\bm{n}'}(\bm{x}'),\\
  \end{split}
\end{equation}
\begin{equation}\label{eR}
\begin{split}
       R_{\bm{n}}(\bm{x})&:=\frac{1}{2}\d_{n_d x_d,\frac1{\n'}}D_{\bm{n}'}(\bm{x}')+\frac{x_d}{2\pi i}\sum_{\nu\neq 0}\frac{1}{\nu(2\pi\nu+x_d)}\d_{n_d(2\pi\nu+x_d),\frac1{\n'}}D_{\bm{n}'}(\bm{x}').
\end{split}
\end{equation}

To formulate and prove the main result, in addition to the above notation
the following one will be helpful. Let $\bm{\eta}$ be a $d$-dimensional vector, with the entries either $0$
or $1$ only. We denote $|\bm{\eta}|=\eta_1+\dots+\eta_d$. Obviously, $0\le|\bm{\eta}|\le d$.

\begin{theorem}\label{mar} Let $\n=(n_1,\dots,n_d)\in \R^d$ and $3<n_1\le n_2\le\dots\le n_d$. We have the asymptotic relation
\begin{equation}\label{mainfor}
\begin{split}
    \mathcal{L}_{\D(\n)}=\Vert D_{\n}\Vert_{(d)}&=\frac{2^{d+1}}{\pi}\bigg(1+\sum_{j=1}^d\frac{\log n_1}{\log n_j}\bigg)\prod_{i=1}^d \log n_i\\
    &+\sum\limits_{k=2}^d \mathfrak{F}_{\n^k}^k\sum\limits_{\bm{\eta}: |\bm{\eta}|=d-k}\,\prod\limits_{i: \eta_i=1}\log\frac{n_{d-i+1}}{n_1}\\
    &+\mathcal{O}\({\log\log n_1}\prod_{j=2}^d \log n_j \),
\end{split}
\end{equation}
where
\begin{equation}\label{fr1}
\begin{split}
     \mathfrak{F}^k_{\n^k}:=&2\pi\sum_{l=0}^{k-2}\Bigg(\Vert F_{n_1 \mathbf{1}^l,\n^{k-l}}\Vert_{(k-1)}-\Vert F_{n_1 \mathbf{1}^l,\n^{k-l-1},n_1}\Vert_{(k-1)}\\
    &+\sum_{1\le |\mu|\le [\frac{n_{k-l}}{n_1}]}\frac1{|\mu|}\int_{-\pi}^{\pi}\(\Vert \d_{n_1(t+2\pi \mu),\frac1{\widetilde{\n}^{(l)}}}
    F_{\widetilde{\n}^{(l)},n_1} \Vert_{(k-2)}-2\Vert F_{\widetilde{\n}^{(l)},n_1}\Vert_{(k-2)}\){\rm d}t\Bigg)
\end{split}
\end{equation}
and
$$
\widetilde{\n}^{(l)}:=(n_1 \bm{1}^l, \n_2^{({k-l-1})})\in \R^{k-2}.
$$
\end{theorem}

\begin{remark}
The formula for $\mathfrak{F}^k_{\n^k}$ looks quite complicated. However, in certain standard and practical cases it can successfully be
estimated (see, in particular, Corollaries~\ref{kuztype} and~\ref{kuztype+} below). Generally, we do not know asymptotic properties of this
quantity. For instance, even in the simpler case of Theorem C the final answer is still unknown; say, in 4) the intriguing problem is whether
the numbers $\omega$ and $\Omega$ are equal or not.
\end{remark}

Allowing more sophisticated relation between directions than in \cite{Daug} and \cite{Ku}, we arrive at

\begin{corollary}\label{kuztype}
Let $\n=(n_1,\dots,n_d)\in \R^d$, $3<n_1\le n_2\le\dots\le n_d$, and  $n_d=\l_j n_j+p_j$, $p_j< n_j$, $j=1,\dots,d-1$, $\l_j,p_j\in \Z_+$. Then
\begin{equation}\label{mainfor1}
  \begin{split}
          \mathcal{L}_{\D(\n)}=&\frac{2^{d+1}}{\pi}\bigg(1+\sum_{j=1}^d\frac{\log n_1}{\log n_j}\bigg)\prod_{i=1}^d \log n_i+\mathcal{O}\({\log\log n_1}\prod_{k=2}^d \log n_k \),
   \end{split}
\end{equation}
where the constant in $\mathcal{O}$ depends only on $d$ and $p_j$ but not on $\l_j$.
\end{corollary}

Combining Theorem~\ref{mar} and Remark~\ref{lenM} below, we readily obtain the following corollary.
\begin{corollary}\label{kuztype+}
Let $\n=(n_1,\dots,n_d)\in \R^d$, $3<n_1\le n_2\le\dots\le n_d$, and $\lim\limits_{n_{d-1}\to\infty} \frac{\log n_d}{\log n_{d-1}}=\infty$. Then
\begin{equation*}
  \begin{split}
    \mathcal{L}_{\D(\n)}&=\frac{2^{d+1}}{\pi}\bigg(1+\sum_{j=1}^{d-1}\frac{\log n_1}{\log n_j}\bigg)\prod_{i=1}^d\log n_i
    +\mathcal{O}\(\Big(\frac{\log n_{d-1}}{\log n_d}+\frac{\log\log n_1}{\log n_1}\Big)\prod_{k=1}^d \log n_k \).
   \end{split}
\end{equation*}
\end{corollary}

\begin{remark}
To make our main result more transparent, we add, in addition to the two-dimensional version in the introduction,
the three-dimensional one, that is, for $d=3$, in notations and assumptions of Theorem~\ref{mar}:

\begin{equation*}
  \begin{split}
       \mathcal{L}_{\D(\n^3)}&=\frac{2^4}{\pi}\(2\log n_1 \log_2 \log n_3+\log n_2 \log^2 n_1+\log n_3 \log^2 n_1\)\\
       &+2\pi \(\Vert F_{n_1,n_2,n_3}\Vert_{(2)}-\Vert F_{n_1,n_2,n_1}\Vert_{(2)}+\Vert F_{n_1,n_1,n_2}\Vert_{(2)}\)\\
       &+2\pi \Vert F_{n_1,n_2} \Vert_{(1)}\(\log\frac{n_2}{n_1}+\log\frac{n_3}{n_1}\)\\
       &+2\pi \sum_{1\le |\mu|\le [\frac{n_3}{n_1}]}\frac1{|\mu|}\int_{-\pi}^\pi \(\Vert \d_{n_1(t+2\pi \mu),\frac1{n_2}}
    F_{n_2,n_1} \Vert_{(1)}-2\Vert F_{n_2,n_1}\Vert_{(1)}\){\rm d}t\\
       &+\mathcal{O}\(\log\log n_1 \log n_2\log n_3\).
   \end{split}
\end{equation*}
\end{remark}

\section{Auxiliary results and proofs}

The proof of Theorem \ref{mar} itself is supposed to be rather short. The burden of
all complicated technicalities will fall on a series of auxiliary lemmas.

\begin{lemma}\label{leE}
The following equality holds
\begin{equation}\label{e19}
  \begin{split}
D_{\bm{n}}(\bm{x})=S_{\bm{n}}(\bm{x})-e^{in_dx_d}F_{ \bm{n}}(\bm{x}'-x_d \bm{m}^{(d-1)})+R_{\bm{n}}(\bm{x}).
  \end{split}
\end{equation}
\end{lemma}

\begin{proof}
  We have
  \begin{equation}\label{e1}
    \begin{split}
    D_{ \n}(\x)&=\sum_{\k'=0}^{[\bm{\L}^{d-1}]}e^{i(\k',\,\x')}
     \int_0^{\L_d(\k')}e^{i\xi x_d} \, {\rm d}[\xi]\\
     &=\sum_{\k'=0}^{[\bm{\L}^{d-1}]}e^{i(\k',\,\x')}
     \int_0^{\L_d(\k')}e^{i\xi x_d} {\rm d} \xi-\sum_{\k'=0}^{[\bm{\L}^{d-1}]}e^{i(\k',\,\x')}
     \int_0^{\L_d(\k')}e^{i\xi x_d}\,  {\rm d}\{\xi\}\\
     &:=S_{ \n}(\x)-I_{ \n}(\x).
  \end{split}
  \end{equation}

Integrating by parts in $I_{ \n}(\x)$, we get
\begin{equation}\label{e16}
  \int_0^{\L_d(\k')}e^{i\xi x_d} \, {\rm d}\{\xi\}=e^{i\L_d(\k')x_d}\{\L_d(\k')\}-ix_d\int_0^{\L_d(\k')}\{\xi\}e^{i\xi x_d} \, {\rm d}\xi.
\end{equation}
Representing $\{\cdot\}$ via the Fourier series as
$$
\{\xi\}=\frac12+\sum_{\nu\neq 0}\frac1{2\pi i \nu}e^{2\pi i\nu \xi},
$$
we obtain
\begin{equation}\label{e17}
  \begin{split}
     &ix_d\int_0^{\L_d(\k')}\{\xi\}e^{i\xi x_d} \, {\rm d}\xi=ix_d\int_0^{\L_d(\bm{k}^{d-1})}\(\frac12+\sum_{\nu\neq 0}\frac1{2\pi i \nu}e^{2\pi i\nu \xi}\)e^{i\xi x_d} \, {\rm d}\xi\\
     &=\frac{e^{\L_d(\k')x_d}-1}{2}+\frac{x_d}{2\pi i}\sum_{\nu\neq 0}\frac{1}{\nu(2\pi\nu+x_d)}\(e^{\L_d(\k')(2\pi\nu+x_d)}-1\).
  \end{split}
\end{equation}
Combining~\eqref{e16} and~\eqref{e17}, we get
\begin{equation*}
  \begin{split}
    I_{\n}(\x)&=e^{in_dx_d}\sum_{\k'=0}^{[\bm{\L}^{d-1}]}\{\L_d(\k')\}e^{i(\k',\,\x'-\bm{m}^{(d-1)}x_d)}\\
    &-\frac{1}{2}\d_{n_d x_d,\frac1{\n'}}D_{\bm{n}'}(\bm{x}')-\frac{x_d}{2\pi i}\sum_{\nu\neq 0}\frac{1}{\nu(2\pi\nu+x_d)}\d_{n_d(2\pi\nu+x_d),\frac1{\n'}}D_{\bm{n}'}(\bm{x}') \\
    &=e^{in_dx_d}F_{ \n}(\x'-\bm{m}^{(d-1)}x_d)+R_{\n}(\x).
  \end{split}
\end{equation*}
This and~\eqref{e1} yield~\eqref{e19}.
\end{proof}

\begin{lemma}\label{thM} 
Let $\n \in (3,+\infty)^d.$ 
Then
\begin{equation}\label{d1}
  \left\Vert D_{\bm{n}}\right\Vert_{(d)}\lesssim \Pi(\n),
\end{equation}
\begin{equation}\label{s1}
  \left\Vert S_{\bm{n}}\right\Vert_{(d)}\lesssim \Pi(\n),
\end{equation}
\begin{equation}\label{e20}
  \Vert R_{\bm{n}}\Vert_{(d)}\lesssim \Pi(\n'), 
\end{equation}
\begin{equation}\label{nM}
 \left\Vert F_{\bm{n}} \right\Vert_{(d-1)}
\lesssim \Pi(\n). 
\end{equation}
\end{lemma}

\begin{proof}
Estimate~\eqref{d1} can be found in~\cite{KL}.
The proof of~\eqref{s1} is the same as the proof of Lemma~3.3 in~\cite{KL}. The proof of~\eqref{e20} follows easily from~\eqref{eR} and~\eqref{d1}.
Inequality~\eqref{nM} follows from \eqref{d1}, \eqref{s1}, ~\eqref{e19}, and periodicity of the functions $D_{\bm{n}}$,
$F_{\bm{n}}$, and $R_{\bm{n}}$ as functions of the variables $x_1,\dots,x_{d-1}$.
\end{proof}

\begin{remark}\label{lenM} 
 Let $\n \in (3,+\infty)^d$. Then
\begin{equation*}
 \left\Vert F_{\bm{n}} \right\Vert_{(d-1)}
\lesssim \log N\, \Pi(\n'),
\end{equation*}
where
$N=\max\{n_1,\dots,n_{d-1}\}$.
\end{remark}

\begin{proof}
The proof of the remark comes as a more accurate proof of Lemma~3.9 in~\cite{KL}.
\end{proof}

\begin{lemma}\label{ld2}
For any $\a,\b\in \R$ and $n\in\N$, the following equality holds:
\begin{equation}\label{eqL}
  \begin{split}
    \int_\T\int_\T |e^{i(\a y+\b)}D_n(x-y)-D_n(x)|\,{\rm d}x\,{\rm d}y
    =4\pi\Vert D_n\Vert_{(1)}+\mathcal{O}(\log\log n),
  \end{split}
\end{equation}
where the constant in $\mathcal{O}$ is independent of $\a$ and $\b$.
\end{lemma}

\begin{proof}
We have
\begin{equation*}
  \int_\T\int_\T |e^{i(\a y+\b)}D_n(x-y)-D_n(x)|\,{\rm d}x\,{\rm d}y=\int_\T \int_{|x|\le\frac1{\log n}}+\int_\T \int_{\frac1{\log n}\le |x|\le \pi}:=J_1+J_2.
\end{equation*}
Taking into account that
$$
\int_\T \int_{|x|\le\frac1{\log n}}|D_n(x-y)|\,{\rm d}x\,{\rm d}y=\int_{|x|\le\frac1{\log n}}\int_\T|D_n(y)|\,{\rm d}x\,{\rm d}y=\mathcal{O}(1),
$$
we derive that
$$
J_1=\int_\T \int_{|x|\le\frac1{\log n}} \(|D_n(x-y)|+|D_n(x)|\)\,{\rm d}x\,{\rm d}y+\mathcal{O}(1).
$$
At the same time, since
$$
\int_\T \int_{\frac1{\log n}\le |x|\le \pi}|D_n(x)|\,{\rm d}x\,{\rm d}y=\mathcal{O}(\log \log n),
$$
we get
$$
J_2=\int_\T \int_{\frac1{\log n}\le |x|\le \pi} \(|D_n(x-y)|+|D_n(x)|\)\,{\rm d}x\,{\rm d}y+\mathcal{O}(\log \log n).
$$
Combining the above estimates for $J_1$ and $J_2$, we derive~\eqref{eqL}.
\end{proof}

\noindent We will also need the following well-known statements (see, e.g.,~\cite[p.~46 and p.~102]{DeLo}).
\begin{lemma}\label{lsp1}
For each continuously differentiable $2\pi$-periodic  $g:\,\mathbb{R}\to \mathbb{R}$ and $h\in\mathbb{R}$
\begin{equation} \label{sp1}
\|g(\,\cdot+h)-g(\,\cdot\,)\|_{L_1(\T)}\leq |h|\|g'\|_{L_1(\T)}.
\end{equation}
 For each trigonometric polynomial $\tau_n$ of degree at most $n$, the Bernstein inequality
\begin{equation} \label{sp2}
\|\tau'_n\|_{L_1(\T)}\leq n\|\tau_n\|_{L_1(\T)}
\end{equation}
holds.
\end{lemma}

The following lemma is an important ingredient of our study. It shows that the norm of the Dirichlet kernel in (\ref{e19})
is the sum of the norms of the first two essential summands, with an appropriate estimate of the remainder term.

\begin{lemma}\label{lemAs1}
Let $\n \in (3,+\infty)^d$. Then
\begin{equation}\label{As1}
  \left\Vert D_{\bm{n}}\right\Vert_{(d)}=\left\Vert S_{\bm{n}}\right\Vert_{(d)}+
  2\pi\left\Vert F_{\bm{n}}\right\Vert_{({d-1})}+\mathcal{O}\(\log\log n_d \cdot \Pi(\n')\).
\end{equation}
\end{lemma}

\begin{proof}
The assertion of the lemma follows from equality~\eqref{e19}, the proof that the norms of the first two summands
are attained on different sets, and proper estimates of the remainder terms. Indeed, using~\eqref{nM}, we get
\begin{equation*}
  \begin{split}
    &\int_{\T^{d-1}}\int_{|x_d|\le \frac1{\log n_d}}|F_{\n}(\x'-x_d \bm{m}^{(d-1)})|\,{\rm d}\x\lesssim \Pi(\n)\int_{|x_d|\le \frac1{\log n_d}}\,{\rm d}x_d\lesssim \Pi(\n'),
  \end{split}
\end{equation*}
which along with~\eqref{e19} and~\eqref{e20} implies
\begin{equation}\label{e22}
  \begin{split}
    &\int_{\T^{d-1}}\int_{|x_d|\le \frac1{\log n_d}}|D_{\n}(\x)|{\rm d}\x=\int_{\T^{d-1}}\int_{|x_d|\le \frac1{\log n_d}}|S_{\n}(\x)|\,{\rm d}\x\\
    &\qquad+\int_{\T^{d-1}}\int_{|x_d|\le \frac1{\log n_d}}|F_{\n}(\x'-\bm{m}^{(d-1)}x_d)|\,{\rm d}\x+\mathcal{O}\(\Pi(\n')\).
  \end{split}
\end{equation}
At the same time, using~\eqref{e2} and~\eqref{d1}, it is not difficult to see that
\begin{equation*}
  \begin{split}
    \int_{\frac1{\log n_d}\le |x_d|\le \pi}\int_{\T^{d-1}}|S_{\n}(\x)|\,{\rm d}\x
    &\le 2\log\log n_d \Vert D_{\n'}\Vert_{({d-1})}\lesssim \log\log n_d\cdot \Pi(\n'),
  \end{split}
\end{equation*}
which, again along with~\eqref{e19} and~\eqref{e20}, implies
\begin{equation}\label{e24}
  \begin{split}
    &\int_{\T^{d-1}}\int_{\frac1{\log n_d}\le |x_d|\le \pi}|D_{\n}(\x)|\,{\rm d}\x
    =\int_{\T^{d-1}}\int_{\frac1{\log n_d}\le |x_d|\le \pi}|S_{\n}(\x)|\,{\rm d}\x\\
    &\qquad+\int_{\T^{d-1}}\int_{\frac1{\log n_d}\le |x_d|\le \pi}|F_{\n}(\x'-x_d\bm{m}^{(d-1)})|\,{\rm d}\x+\mathcal{O}\(\log\log n_d\cdot \Pi(\n')\).
  \end{split}
\end{equation}
Finally, combining~\eqref{e22} and \eqref{e24}, we obtain~\eqref{As1}.
\end{proof}

Crucial as the previous lemma may seem, the next one is even more important. It is
our decisive tool, in a sense.

\begin{lemma}\label{main}
Let $\n \in (3,+\infty)^d$ and $n_1\le n_2\le \dots \le n_d$.  Then
\begin{equation}\label{eM}
\begin{split}
    \Vert D_{\n}\Vert_{(d)}=\Vert D_{n_1 \mathbf{1}^d}\Vert_{(d)}&+2\sum_{\nu=1}^{d-1}\log\frac{n_{d-\nu}}{n_1} \Vert D_{n_1 \mathbf{1}^\nu,\n^{d-\nu}}\Vert_{(d-1)}\\
    &+\mathfrak{F}^d_{\n}+\mathcal{O}\(\log\log n_1 \Pi(\n_2)\),
\end{split}
\end{equation}
where $\mathfrak{F}_{\n^d}^d$ is defined in~\eqref{fr1}.
\end{lemma}

\begin{proof}
In (\ref{As1}), we first consider the norm  $\left\Vert S_{\n}\right\Vert_{(d)}$. Substituting $x_d \mapsto \frac{n_1}{n_d}x_d$, we get
\begin{equation}\label{m1}
  \begin{split}
    \left\Vert S_{\n}\right\Vert_{(d)}&=\int_{|x_d|\le \frac{n_d}{n_1}\pi} \left\Vert S_{\n',n_1}(\cdot,x_d)\right\Vert_{(d-1)}{\rm d}x_d\\
    &=\left\Vert S_{\n',n_1}\right\Vert_{(d)}+\int_{\pi\le|x_d|\le \frac{n_d}{n_1}\pi} \left\Vert S_{\n',n_1}(\cdot,x_d)\right\Vert_{(d-1)}{\rm d}x_d:=J_1+J_2.
  \end{split}
\end{equation}
Applying Lemma~\ref{lemAs1} to $J_1$ and using the fact that $\left\Vert D_{\n',n_1}\right\Vert_{(d)}=\left\Vert D_{n_1,\n'}\right\Vert_{(d)}$, we obtain
\begin{equation}\label{m2}
  J_1=\left\Vert D_{n_1,\n'}\right\Vert_{(d)}-2\pi\left\Vert F_{\n',n_1}\right\Vert_{(d-1)}+\mathcal{O}\(\log\log n_1 \Pi(\n')\).
\end{equation}
For $J_2$, using~\eqref{e2} and~\eqref{d1}, we derive
\begin{equation}\label{m3}
  \begin{split}
     J_2&=\sum_{1\le |k|\le [\frac{n_d}{2n_1}]} \int_{(2k-1)\pi}^{(2k+1)\pi}
     \left\Vert S_{\n',n_1}(\cdot,x_d)\right\Vert_{(d-1)}{\rm d}x_d+\mathcal{O}\(\Pi(\n')\)\\
     &=\sum_{1\le |k|\le [\frac{n_d}{2n_1}]} \int_{(2k-1)\pi}^{(2k+1)\pi}\frac{{\rm d}x_d}{|x_d|}
     \big\Vert \d_{n_1x_d,\frac1{\n'}}D_{\n'}\big\Vert_{(d-1)}+\mathcal{O}\(\Pi(\n')\)\\
     &=\sum_{1\le |k|\le [\frac{n_d}{2n_1}]}\frac1{2\pi |k|} \int_{-\pi}^{\pi}
     \big\Vert \d_{n_1(x_d+2\pi k),\frac1{\n'}} D_{\n'}\big\Vert_{(d-1)}\,{{\rm d}x_d}+\mathcal{O}(\Pi(\n')).\\
  \end{split}
\end{equation}
Thus, applying Lemma~\ref{ld2}, we prove the lemma in the case $d=2$.

In the reset of the proof, we consider only the case $d\ge 3$. Denoting by $\d_{x_d,k}^*$ the operator
$$
\d_{x_d,k}^* f(\x_2',x_1):=e^{in_1(x_d+2\pi k)}f\(\x_2'-\frac{n_1}{\n_2'}(x_d+2\pi k),x_1-\frac{x_d}{2}\)-f\(\x_2',x_1+\frac{x_d}{2}\),
$$
where $f$ is an arbitrary function of variables $x_2,\dots,x_{d-1},x_1$, and taking into account the periodicity of $D_{\n'}$, it is easy to see that
$$
\d_{n_1(x_d+2\pi k),\frac1{\n'}} D_{\n'}\(\x'+\frac{x_d}2 {\rm e}_1^{d-1}\)=\d_{x_d,k}^* D_{\n_2',n_1}(\x_2',x_1).
$$
Thus, using Lemma~\ref{leE}, we derive
\begin{equation}\label{m4}
\begin{split}
I_k&:=\int_{-\pi}^{\pi}
     \big\Vert \d_{n_1(x_d+2\pi k),\frac1{\n'}} D_{\n'}\big\Vert_{(d-1)}\, {{\rm d}x_d}\\
     &=\int_{-\pi}^{\pi}\big\Vert \d_{n_1(x_d+2\pi k),\frac1{\n'}} D_{\n'}\big(\cdot+\frac{x_d}2 {\rm e}_1^{d-1}\big)\big\Vert_{(d-1)}\,{{\rm d}x_d}\\
     &=\int_{-\pi}^{\pi}\big\Vert \d_{x_d,k}^* D_{\n_2',n_1}\big\Vert_{(d-1)}\,{{\rm d}x_d}=\int_{-\pi}^{\pi}{{\rm d}x_d}\int_{\T^{d-1}}\bigg|\d_{x_d,k}^* S_{\n_2',n_1}(\x_2',x_1)\\
&\qquad\qquad\qquad\qquad-\d_{x_d,k}^* \Big(e^{in_1 x_{1}}F_{\n_2',n_1}\big(\x_2'-x_1\frac{n_1}{\n_2'}\big)\Big)+\d_{x_d,k}^* R_{\n_2',n_1}(\x_2',x_1) \bigg|\, {\rm d}\x'.
\end{split}
\end{equation}
In the inner integral on the right-hand side, let us study each summand separately.

1)  Equality~\eqref{eR} and \eqref{d1} easily yield
\begin{equation}\label{m5}
  \begin{split}
\int_{-\pi}^{\pi}{{\rm d}x_d}\int_{\T^{d-1}}\big|\d_{x_d,k}^* R_{\n_2',n_1}(\x_2',x_1) \big|\, {\rm d}\x'
\lesssim \Vert D_{\n_2'}\Vert_{(d-2)}=\mathcal{O}\(\Pi(\n_2')\).
  \end{split}
\end{equation}

2)  Denoting $\Omega_{n_1}(x_d):=\{x_1\in \T\,:\, |x_1-\frac{x_d}{2}|<\frac1{\log n_1},\,\, |x_1+\frac{x_d}{2}|<\frac1{\log n_1}\}$
and using~\eqref{nM}, we derive
\begin{equation}\label{m6}
  \begin{split}
&\int_{-\pi}^{\pi}{{\rm d}x_d}\int_{\Omega_{n_1}(x_d)}\,{{\rm d}x_1}\int_{\T^{d-2}}\Big|\d_{x_d,k}^* \Big(e^{in_1 x_{1}}F_{\n_2',n_1}\Big(\x_2'-x_1\frac{n_1}{\n_2'}\Big)\Big) \Big|\, {\rm d}\x_2'\\
&\le 2\int_{-\pi}^{\pi}{{\rm d}x_d}\int_{\Omega_{n_1}(x_d)}\,{{\rm d}x_1} \Vert F_{\n_2',n_1}\Vert_{(d-2)}\lesssim \frac1{\log n_1} \Vert F_{\n_2',n_1}\Vert_{(d-2)} =\mathcal{O}\(\Pi(\n_2')\).
  \end{split}
\end{equation}

3) Further, applying~\eqref{d1}, we obtain
\begin{equation}\label{m7}
  \begin{split}
&\int_{-\pi}^{\pi}{{\rm d}x_d}\int_{\T\setminus\Omega_{n_1}(x_d)}{{\rm d}x_1}\int_{\T^{d-2}}\big|\d_{x_d,k}^* S_{\n_2',n_1}(\x_2',x_1) \big| {\rm d}\x_2'\\
&\le \int_{-\pi}^{\pi}{{\rm d}x_d}\int_{\frac1{\log n_1}<|x_1-\frac{x_d}{2}|}{{\rm d}x_1}\int_{\T^{d-2}}\Big|S_{\n_2',n_1}\Big(\x_2'-\frac{n_1}{\n_2'}(x_d+2\pi k),x_1-\frac{x_d}{2}\Big) \Big| {\rm d}\x_2'\\
&\quad+\int_{-\pi}^{\pi}{{\rm d}x_d}\int_{\frac1{\log n_1}<|x_1+\frac{x_d}{2}|}{{\rm d}x_1}\int_{\T^{d-2}}\Big|S_{\n_2',n_1}\(\x_2',x_1+\frac{x_d}{2}\) \Big| {\rm d}\x_2'\\
&\le 2\int_{-\pi}^{\pi}{{\rm d}x_d}\int_{\frac1{\log n_1}<|x_1-\frac{x_d}{2}|}{\frac{{\rm d}x_1}{|x_1-\frac{x_d}{2}|}}\int_{\T^{d-2}}|D_{\n_2'}\(\x_2'\)| {\rm d}\x_2'\\
&\quad+2\int_{-\pi}^{\pi}{{\rm d}x_d}\int_{\frac1{\log n_1}<|x_1+\frac{x_d}{2}|}{\frac{{\rm d}x_1}{|x_1+\frac{x_d}{2}|}}\int_{\T^{d-2}}|D_{\n_2'}\(\x_2'\)| {\rm d}\x_2'=\mathcal{O}(\log\log n_1\Pi(\n_2')).
  \end{split}
\end{equation}

Combining \eqref{m4} and estimates~\eqref{m5}--\eqref{m6} and applying the same arguments as in the proof of Lemma~\ref{lemAs1}, we arrive at the asymptotic equality
\begin{equation}\label{m8}
  \begin{split}
    I_k&=\int_{-\pi}^{\pi}\big\Vert \d_{x_d,k}^* S_{\n_2',n_1}\big\Vert_{(d-1)}\,{{\rm d}x_d}\\
    &\qquad\qquad+2\pi \int_{-\pi}^{\pi}\big\Vert \d_{x_d,k}^* F_{\n_2',n_1} \big\Vert_{(d-2)}\,{\rm d}x_d+\mathcal{O}(\log\log n_1\Pi(\n_2')).
  \end{split}
\end{equation}
Let us estimate the first integral in~\eqref{m8}. First, we observe that
\begin{equation}\label{m9}
  \begin{split}
&\int_{-\pi}^{\pi}\big\Vert \d_{x_d,k}^* S_{\n_2',n_1}\big\Vert_{(d-1)}\,{{\rm d}x_d}\\
&=\int_{-\pi}^{\pi} {\rm d}x_d\int_{-\pi+\frac{x_d}{2}}^{\pi+\frac{x_d}{2}}\,{\rm d}x_1 \int_{\T^{d-2}}
\Big| e^{in_1(x_d+2\pi k)}S_{\n_2',n_1}\Big(\x_2'-\frac{n_1}{\n_2'}(x_d+2\pi k),x_1-x_d\Big)\\
&\qquad\qquad\qquad\qquad\qquad\qquad\qquad\qquad\qquad\qquad\qquad\qquad\qquad-S_{\n_2',n_1}(\x_2',x_1)\Big|\,{\rm d}\x_2'.
  \end{split}
\end{equation}
Furthermore,
\begin{equation}\label{m10}
  \begin{split}
      &\int_{-\pi}^\pi {\rm d}x_d\int_{\frac1{\log n_1}<|x_1|\le \frac{3\pi}{2}}{\rm d}x_1\int_{\T^{d-2}}|S_{\n_2',n_1}(\x_2',x_1)|\,{\rm d}\x_2'\\
      &\le 4\pi\int_{\frac1{\log n_1}<|x_1|\le \frac{3\pi}{2}}\frac{{\rm d}x_1}{|x_1|}\int_{\T^{d-2}}|D_{\n_2'}(\x_2')|\,{\rm d}\x_2'
      =\mathcal{O}\(\log\log n_1 \Pi(\n_2')\)
  \end{split}
\end{equation}
and
\begin{equation}\label{m11}
  \begin{split}
      &\int_{-\pi}^\pi{\rm d}x_d\int_{|x_1|<\frac1{\log n_1}}\,{\rm d}x_1\int_{\T^{d-2}}\Big|S_{\n_2',n_1}\Big(\x_2'-(x_d+2\pi k)\frac{n_1}{\n_2'},x_1-x_d\Big)\Big|\,{\rm d}\x_2'\\
      &=\int_{-\pi}^\pi{\rm d}x_d\int_{|x_1|<\frac1{\log n_1}}\,{\rm d}x_1\int_{\T^{d-2}}\big|S_{\n_2',n_1}\(\x_2',x_1-x_d\)\big|\,{\rm d}\x_2'\\
      &=\int_{|x_1|<\frac1{\log n_1}}{{\rm d}x_1}\int_{|x_1-x_d|<\frac1{n_1}}{\rm d}x_d \int_{\T^{d-2}}|\dots| {\rm d}\x_2'\\
      &\qquad\qquad\qquad\qquad+\int_{|x_1|<\frac1{\log n_1}}{{\rm d}x_1}\int_{\underset{x_d\in [-\pi,\pi]}{\frac1{n_1}<|x_1-x_d|}}{\rm d}x_d \int_{\T^{d-2}}|\dots| {\rm d}\x_2'
      :=L_1+L_2.
  \end{split}
\end{equation}
Let us consider $L_2$. As above, using~\eqref{e2} and~\eqref{d1}, we derive
\begin{equation}\label{m12}
  \begin{split}
L_2 \le 2\int_{|x_1|<\frac1{\log n_1}}{{\rm d}x_1}\int_{\underset{x_d\in [-\pi,\pi]}{\frac1{n_1}<|x_1-x_d|}}
\frac{{\rm d}x_d}{|x_1-x_d|} \int_{\T^{d-2}} |D_{\n_2'}\(\x_2'\)|\,{\rm d}\x_2'=\mathcal{O}\(\Pi(\n_2')\).
  \end{split}
\end{equation}
Further, estimating $L_1$, we obtain
\begin{equation}
\label{m13}
\begin{split}
L_1 &\le  \int_{|x_1|<\frac1{\log n_1}}{{\rm d}x_1}\int_{|x_1-x_d|<\frac1{n_1}}{\rm d}x_d\\
&\qquad\qquad\qquad\qquad\qquad\qquad\int_{\T^{d-2}} \Big|\frac{e^{in_1(x_1-x_d)}-1}{x_1-x_d}
D_{\n_2'}\Big(\x_2'-(x_1-x_d)\frac{n_1}{\n_2'}\Big) \Big|\, {\rm d}\x_2'\\
&+\int_{|x_1|<\frac1{\log n_1}}{{\rm d}x_1}\int_{|x_1-x_d|<\frac1{n_1}}{\rm d}x_d \int_{\T^{d-2}}
\bigg|\frac{D_{\n_2'}\big(\x_2'-(x_1-x_d)\frac{n_1}{\n_2'}\big)-D_{\n_2'}\(\x_2'\)}{x_1-x_d}\bigg|\,{\rm d}\x_2'\\
&:=L_{11}+L_{12}.
\end{split}
\end{equation}
Using the boundedness of the function $\frac{e^{it}-1}{t}$ and \eqref{d1}, we get
\begin{equation}\label{m14}
  \begin{split}
L_{11}=\mathcal{O}\(\frac1{\log n_1} \Pi(\n_2')\).
  \end{split}
\end{equation}
To estimate $L_{12}$, some preparation is needed. For $k=0,\dots,d-3$, we denote
$$
\bm{\mu}^{(k)}:=\(\frac{n_1}{n_2},\dots,\frac{n_1}{n_{d-1-k}},0,\dots,0\)\in \R^{d-2}
$$
and $\bm{\mu}^{(d-2)}:=(0,\dots,0)\in \R^{d-2}$. Using the representation
\begin{equation*}
  \begin{split}
&D_{\n_2'}\(\x_2'-(x_1-x_d)\frac{n_1}{\n_2'}\)-D_{\n_2'}\(\x_2'\)\\
&\qquad\qquad=\sum_{k=0}^{d-3}\(D_{\n_2'}\(\x_2'-(x_1-x_d)\bm{\mu}^{(k)}\)-D_{\n_2'}\(\x_2'-(x_1-x_d)\bm{\mu}^{(k+1)}\)\)
   \end{split}
\end{equation*}
and inequalities~\eqref{sp1} and~\eqref{sp2}, we derive
\begin{equation}\label{m15}
  \begin{split}
L_{12}&\le \sum_{k=0}^{d-3} \int_{|x_1|<\frac1{\log n_1}}{{\rm d}x_1}\int_{|x_1-x_d|<\frac1{n_1}}\,{\rm d}x_d \\ &\qquad\qquad  \int_{\T^{d-2}}\bigg|\frac{D_{\n_2'}\(\x_2'-(x_1-x_d)\bm{\mu}^{(k)}\)-D_{\n_2'}\(\x_2'-(x_1-x_d)\bm{\mu}^{(k+1)}\)}{x_1-x_d}\bigg|\,{\rm d}\x_2'\\
&\le  \sum_{k=0}^{d-3} \int_{|x_1|<\frac1{\log n_1}}{{\rm d}x_1}\int_{|x_1-x_d|<\frac1{n_1}}\, {\rm d}x_d \int_{\T^{d-2}} \frac{n_1}{n_{d-1-k}}
\Big|\frac{\partial}{\partial x_{d-1-k}} D_{\n_2'}(\x_2')\Big|\,{\rm d}\x_2'\\
&\le \frac{(d-2)}{\log n_1} \Vert D_{\n_2'}\Vert_{(d-2)}=\mathcal{O}\(\frac1{\log n_1}\Pi(\n_2')\).
   \end{split}
\end{equation}
Combining~\eqref{m11}--\eqref{m15}, we obtain
\begin{equation}\label{m16}
  \begin{split}
&\int_{-\pi}^\pi {\rm d}x_d\int_{|x_1|<\frac1{\log n_1}}\,{\rm d}x_1\int_{\T^{d-2}}\Big|S_{\n_2',n_1}\Big(\x_2'-(x_d+2\pi k)\frac{n_1}{\n_2'},x_1-x_d\Big)\Big|\,{\rm d}\x_2'\\
&\qquad =\mathcal{O}\(\Pi(\n_2')\).
   \end{split}
\end{equation}
Further, combining~\eqref{m9}, \eqref{m10}, and~\eqref{m16} and using the same arguments as in the proof of Lemma~\ref{ld2},
we see that
\begin{equation*}
  \begin{split}
&\int_{-\pi}^{\pi}\Vert \d_{x_d,k}^* S_{\n_2',n_1}\Vert_{(d-1)}\,{{\rm d}x_d}=4\pi \Vert S_{\n_2',n_1}\Vert_{(d-1)}+\mathcal{O}\(\log\log n_1 \Pi(\n_2')\).
  \end{split}
\end{equation*}
The above asymptotic equality and inequality~\eqref{m8} imply that
\begin{equation}\label{m18}
  \begin{split}
    I_k&=4\pi \big\Vert S_{\n_2',n_1}\big\Vert_{(d-1)}\\
    &\qquad\qquad+2\pi \int_{-\pi}^{\pi}\big\Vert \d_{x_d,k}^* F_{\n_2',n_1} \big\Vert_{(d-2)}{{\rm d}x_d}+\mathcal{O}\(\log\log n_1\Pi(\n_2')\)\\
    &=4\pi \Vert D_{\n'}\Vert_{(d-1)}+2\pi \int_{-\pi}^{\pi}\(\big\Vert \d_{n_1(x_d+2\pi k),\frac1{\n_2'}}
    F_{\n_2',n_1} \big\Vert_{(d-2)}-2\big\Vert F_{\n_2',n_1}\big\Vert_{(d-2)}\)\,{{\rm d}x_d}\\
    &\qquad\qquad\qquad\qquad\qquad\qquad\qquad\qquad\qquad\qquad\qquad\qquad+\mathcal{O}\(\log\log n_1\Pi(\n_2')\)\\
    &=4\pi \Vert D_{\n'}\Vert_{(d-1)}+2\pi\mathfrak{f}_k+\mathcal{O}\(\log\log n_1\Pi(\n_2')\),
  \end{split}
\end{equation}
where
$$
\mathfrak{f}_k:=\int_{-\pi}^{\pi}\(\big\Vert \d_{n_1(x_d+2\pi k),\frac1{\n_2'}} F_{\n_2',n_1} \big\Vert_{(d-2)}-2\big\Vert F_{\n_2',n_1}\big\Vert_{(d-2)}\)\,{dx_d}.
$$
Combining~\eqref{m18} and~\eqref{m3} yields
\begin{equation}\label{m19}
  \begin{split}
     J_2&=2\log \frac{n_d}{n_1}\Vert D_{\n'}\Vert_{(d-1)} +\sum_{1\le |k|\le [\frac{n_d}{2n_1}]}\frac{\mathfrak{f}_k}{|k|} +\mathcal{O}(\log\log n_1\Pi(\n_2)).\\
  \end{split}
\end{equation}
Now, combining~\eqref{m1}, \eqref{m2}, and~\eqref{m19}, we derive
\begin{equation*}
  \begin{split}
  \Vert S_{\n}\Vert_{(d)}=\Vert D_{n_1,\n'}\Vert_{(d)}+2\log\frac{n_d}{n_1}\Vert D_{\n'}\Vert_{(d-1)}&-2\pi\Vert F_{\n',n_1}\Vert_{(d-1)}\\
  &+\sum_{1\le |k|\le [\frac{n_d}{2n_1}]}\frac{\mathfrak{f}_k}{|k|} +\mathcal{O}\(\log\log n_1\Pi(\n_2)\).
  \end{split}
\end{equation*}
Together with~\eqref{As1} and~\eqref{nM}, this implies
\begin{equation*}
  \begin{split}
 \Vert D_{\n}\Vert_{(d)}=\Vert D_{n_1,\n'}\Vert_{(d)}&+2\log\frac{n_d}{n_1}\Vert D_{\n'}\Vert_{(d-1)}+2\pi\(\Vert F_{\n}\Vert_{(d-1)}-\Vert F_{\n',n_1}\Vert_{(d-1)}\)\\
  &+\sum_{1\le |k|\le [\frac{n_d}{n_1}]}\frac{\mathfrak{f}_k}{|k|} +\mathcal{O}\(\log\log n_d \Pi(\n')\).
  \end{split}
\end{equation*}
Repeating the above reasoning $d-2$-times in other variables, we obtain~\eqref{eM}.
\end{proof}

We are now in a position to complete the proof of Theorem~\ref{mar}, mainly by incorporating the above lemmas
in an appropriate way.

\begin{proof}[Proof of Theorem~\ref{mar}] The proof is inductive. By Lemma \ref{main}, we have
\begin{equation*}
\begin{split}
    \Vert D_{\n}\Vert_{(d)}=\Vert D_{n_1\mathbf{1}^d}\Vert_{(d)}&+2\sum_{k=1}^{d-1}\log\frac{n_{d-k}}{n_1} \big\Vert D_{n_1\mathbf{1}^k,\n_2^{d-k}}\big\Vert_{({d-1})}\\
    &+\mathfrak{F}_{\n^d}^d+\mathcal{O}\(\log\log n_1 \Pi(\n_2)\),
\end{split}
\end{equation*}
Applying now the $(d-1)$-dimensional version of (\ref{mainfor}) to each of the summands in the sum
of the right-hand side and taking into account that some of the logarithmic bounds are controlled
by the ''stronger" ones, we arrive at
\begin{equation*}
\begin{split}
    \Vert D_{\n}\Vert_{(d)}&=\Vert D_{n_1\mathbf{1}^d}\Vert_{(d)}+\sum_{k=1}^{d-1}2^k \sum\limits_{\bm{\eta}: |\bm{\eta}|=k} \prod\limits_{j: \eta_j=1}
    \log\frac{n_{j}}{n_1} \big\Vert D_{n_1\mathbf{1}^{d-k}}\big\Vert_{({d-k})}\\
    &+\sum\limits_{k=2}^d \mathfrak{F}_{\n^k}^k\sum\limits_{\bm{\eta}: |\bm{\eta}|=d-k}\,\prod\limits_{i: \eta_i=1}\log\frac{n_{d-i+1}}{n_1}
    +\mathcal{O}\({\log\log n_1}\Pi(\n_2) \).
\end{split}
\end{equation*}
Then, applying asymptotic formula~\eqref{eS} and simple calculations, we obtain the required result.
\end{proof}

We now proceed to the proof of Corollary \ref{kuztype}.

\begin{proof} One can see that (\ref{mainfor1}) is (\ref{mainfor}) without any $\mathfrak{F}_{\n^k}^k$ and accompanied factors.
To establish this, it suffices to prove that the $L_1$ norm of $F_{\bm{n}}(\bm{x}')$ is of smaller growth than the remainder terms.
Putting  $n_d=\l_j n_j+p_j$ in $\{\L_d(\bm{k}')\}$ reduces the latter to $\{\sum\limits_{j=1}^{d-1} \frac{p_j k_j}{n_j}\}$.
Each $\frac{p_j k_j}{n_j}\in [l_j,l_j+1)$, where $l_j$ varies from $0$ to $[\L_j(\k^{j-1})]-1$. For each $(k_1,\dots,k_{d-1})$, there is
$(l_1,\dots,l_{d-1})$ such that
$$
l_1+\dots+l_{d-1} \le \sum\limits_{j=1}^{d-1} \frac{p_j k_j}{n_j}<l_1+\dots+l_{d-1}+1
$$
for certain values of $l_1,\dots,l_{d-1}$. Correspondingly,
$$
F_{\bm{n}}(\bm{x}')=\sum_{\bm{k}'=0}^{[\bm{\L}^{d-1}]}\left(\frac{p_1k_1}{n_1}-l_1+\dots+\frac{p_{d-1}k_{d-1}}{n_{d-1}}-l_{d-1}\right) e^{i(\bm{k}',\,\bm{x}')}.$$
The $L_1(\mathbb T^{d-1})$ norm is now dominated by the sum of the $L_1(\mathbb T^{d-1})$ norms of
\begin{equation}\label{sepjn}
\sum_{\bm{k}'=0}^{[\bm{\L}^{d-1}]}\left(\frac{p_jk_j}{n_j}-l_j\right)e^{i(\bm{k}',\,\bm{x}')},
\end{equation}
for $j=1,\dots,d-1$, each of which, in turn, is bounded by two norms. One is $\frac{p_j}{n_j}\|
\frac{\partial}{\partial x_j}D_{ \bm{n}'}\|_{(d-1)}$. By (\ref{sp2}), it is dominated by $p_j
\|D_{ \bm{n}'}\|_{(d-1)}$. By (\ref{d1}), it is $\mathcal{O}(\log n_1\dots\log n_{d-1})$. For the second norm coming from (\ref{sepjn}), we take it of
$$
\sum\limits_{k_1=0}^{\L_1}\dots\sum\limits_{k_{j-1}=0}^{\L_{j-1}(\bm{k}^{j-2})}
 \sum\limits_{l_j=0}^{p_j-1} \sum\limits_{k_j=[l_jn_j/p_j]}^{[(l_j+1)n_j/p_j]} l_j
\sum\limits_{k_{j+1}=0}^{\L_{j+1}(\bm{k}^{j})}\dots
\sum\limits_{k_{d-1}=0}^{\L_{d-1}(\bm{k}^{d-2})}e^{i(\bm{k}',\,\bm{x}')}.
$$
Reordering summation so that the $j$-th some be on the first place and recounting the inner sum in the corresponding summand
as running from $0$ to $[(l_j+1)n_j/p_j]-[l_jn_j/p_j]$, we derive that the corresponding norm is dominated by
$p_j^2 \log n_1\dots\log n_{j-1}\log n_{j+1}\dots\log n_{d-1}$. All these bounds definitely enter in the remainder term of (\ref{mainfor}), as desired.
\end{proof}

\section*{Acknowledgements}

This project started during the visit of the first author to the Gelbart Institute at Department
of Mathematics in Bar-Ilan University, Ramat-Gan, Israel, and continued when the second authors visited
the University of L\"ubeck. The authors appreciate the hospitality and stimulating atmosphere
they enjoyed during these visits. The authors are indebted to A.N.~Podkorytov for his valuable comments
which helped to simplify the formulation of our main results. The first author was partially supported by DFG project KO 5804/1-1.


\begin{thebibliography}{99}
{

\bibitem{Ash} J.M. Ash, {\it Triangular Dirichlet Kernels and Growth of $L^p$ Lebesgue Constants}, J. Fourier Anal. Appl.  \textbf{16}, no. 6 (2010), 1053--1069.

\bibitem{AC} J.M. Ash, L. De Carli, {\it Growth of $L^p$ Lebesgue constants for convex polyhedra and other regions},
Trans. Amer. Math. Soc. \textbf{361} (2009), 4215--4232.

\bibitem{Ba} S.P. Baiborodov, {\it Lebesgue Constants of Polyhedra}, Mat. Zametki \textbf{32} (1982), 817--822 (Russian); English translation in Math. Notes \textbf{32} (1982), 895--898.

\bibitem{Be} E.S. Belinsky, {\it Behavior of the Lebesgue constants of certain methods of summation of
multiple Fourier series}, Metric Questions of the Theory of Functions and Mappings, Naukova Dumka, Kiev, 1977, 19--39 (Russian).


\bibitem{Daug} I.K. Daugavet, {\it On the Lebesgue constants for double Fourier series}, Methods of Computations, Leningrad Univ.
\textbf{6} (1970), 8--13 (Russian).

\bibitem{DMV} B. Della Vecchia, G. Mastroianni, P. V\'ertesi, {\it A remark on multivariate projection operators},
Acta Math. Hungar. {\bf 139} (2013), 337--343.

\bibitem{DEKL} P. Dencker, W. Erb, Yu. Kolomoitsev, T. Lomako, {\it Lebesgue constants for polyhedral sets and polynomial
interpolation on Lissajous-Chebyshev nodes}, J. of Complexity \textbf{43} (2017), 1--27.

\bibitem{DeLo} {R.A. DeVore, G.G. Lorentz}, {\it Constructive Approximation}, Springer-Verlag, New York, 1993.

%


\bibitem{GaLi} M. Ganzburg, E. Liflyand, {\it The Lebesgue Constants
of Fourier Partial Sums}, Topics in Classical and Modern Analysis, Birkh\"auser, 2019, 147--158.

\bibitem{KL} Yu. Kolomoitsev, T. Lomako, {\it On the growth of Lebesgue constants for convex polyhedra}, Trans. Amer. Math. Soc. \textbf{370} (2018), 6909--6932.

\bibitem{KS}  S.V.  Konyagin, M.A. Skopina, {\it Comparison of the $L^1$-norms of total and truncated exponential sums},
Mat. Zametki \textbf{69} (2001), 699--707 (Russian);  English translation in  Math. Notes \textbf{69} (2001), 644--651.

\bibitem{Ku} O.I. Kuznetsova, {\it The Asymptotic Behavior of the Lebesgue Constants for a Sequence of
Triangular Partial Sums of Double Fourier Series}, Sib. Mat. Zh. \textbf{XVIII} (1977),
629--636 (Russian); English translation in Siberian Math. J. \textbf{18} (1977), 449--454.

\bibitem{L} E.R. Liflyand, {\it Lebesgue Constants of multiple Fourier series}, Online J. Anal. Comb. \textbf{1}, no. 5 (2006), 1--112.

\bibitem{NP} F. Nazarov, A. Podkorytov, {\it On the behavior of the Lebesgue constants for two dimensional
Fourier sums over polygons}, Algebra i Analiz \textbf{7} (1995), 214--238 (Russian); English translation in St.-Petersburg Math.~J. \textbf{7} (1995), 663--680.

\bibitem{P} A.N. Podkorytov, {\it The order of growth of the Lebesgue constants of Fourier sums over polyhedra},
Vestnik Leningrad. Univ. Matem. \textbf{7} (1982), 110--111 (Russian).

\bibitem{P2} A.N. Podkorytov, {\it On asymptotics of Dirichlet's kernels of Fourier sums with respect to a polygon},
Investigations on linear operators and function theory. Part XV, Zap. Nauchn. Sem. LOMI \textbf{149}, Nauka, Leningrad. Otdel., Leningrad (1986), 142--149.

\bibitem{S} M.A. Skopina, {\it Lebesgue constants of multiple polyhedron sums of de la Vallee–Poussin}, Analytical
theory of numbers and theory of functions. Part 5, Zap. Nauchn. Sem. LOMI \textbf{125}, Nauka, Leningrad. Otdel., Leningrad, 1983, 154--165.

\bibitem{SV} L. Szili, P. V\'ertesi, {\it On multivariate projection operators}, J. Approx. Theory {\bf 159} (2009), 154--164.

\bibitem{TB} R.M. Trigub, E.S. Belinsky, {\it Fourier Analysis and Appoximation of Functions}, Kluwer, 2004.

\bibitem{YuYu} A.A. Yudin, V.A. Yudin, {\it Polygonal Dirichlet kernels and growth of Lebesgue constants}, Mat. Zametki
\textbf{37} (1985), 220--236  (Russian); English translation in Math. Notes \textbf{37} (1985),  124--135.


}

\end{thebibliography}
\end{document}